\documentclass[12pt,textwidth=125 mm,
textheight=195 mm]{amsart}
\usepackage{amsrefs, amsthm}
\usepackage{amsmath}
\usepackage[capitalize]{cleveref}

\renewcommand{\subset}{\subseteq}
\renewcommand{\root}{r}
\newcommand{\NN}{\mathbb{N}}
\newcommand{\ZZ}{\mathbb{Z}}

\newtheorem{theorem}{Theorem}
\newtheorem{proposition}[theorem]{Proposition}
\newtheorem{lemma}[theorem]{Lemma}
\newtheorem{corollary}[theorem]{Corollary}

\theoremstyle{remark}
\newtheorem{remark}[theorem]{Remark}

\begin{document}
\title{The nodal cubic and
quantum groups at roots of unity}
\author{Ulrich Kr\"ahmer}
\author{Manuel Martins}
\address{TU Dresden, Institute of Geometry,
01062 Dresden}
\email{ulrich.kraehmer@tu-dresden.de,\newline
 \phantom{Eilil addressii} 
manuel.martins@tu-dresden.de}
\begin{abstract}
In a recent article, the coordinate ring of
the nodal cubic was given the structure of a
quantum homogeneous space. Here the
corresponding coalgebra Galois extension 
is expressed in terms of quantum
groups at roots of unity, and is shown to be
cleft. Furthermore, the minimal quotient
extensions are determined.
\end{abstract}
\thanks{M.~M. is supported by FCT (Portugal) grant 
SFRH/BD/128622/2017. It is our pleasure to thank Ken Brown,
Angela Tabiri and the referees for their 
comments, suggestions and
discussion.}

\maketitle

\section{Introduction}
In \cite{kt}, the
coordinate ring $B$ of the nodal cubic
is embedded as a right coideal subalgebra
into a Hopf algebra which 
is free over $B$. 
In this paper, we determine a minimal Hopf algebra $A$ with this property,
and show that  
the Galois extension $B \subset A$ is
cleft. The main result is:

\begin{theorem}\label{maintheorem}
\label{thm introduction}
Let $k$ be a field containing a primitive
third root $\root$ of unity, 
and $A$ be the Hopf algebra  
over $k$ generated 
by $a, b,x$ and $y$ satisfying 
\begin{gather*}
ba = ab,\qquad 
ya = ay,\qquad 
bx = xb,\qquad
yx = xy,\qquad
by = -yb,\\
	x + axa^2 + a^2xa - a + 1 = 0,\qquad
	x^2 + ax^2a^2 + xaxa^2 + x + axa^2 = 0,\\
y^{2} = x^{2} + x^{3},\qquad
b^{2} = a^{3} = 1 ,\qquad
	F^3=0,
\end{gather*}
where $F:=xa + (r+1)ax + \frac{r+2}{3} 
\left(a - a^2\right)$, and 
\begin{gather*}
\Delta (a)=a \otimes a,\qquad
\Delta (b) = b \otimes b,\\
\Delta (x) = 1 \otimes x+x \otimes a,\qquad
\Delta (y) = 1 \otimes y+y \otimes b.
\end{gather*}
Let $p,q \in k$ satisfy
$p^2=q^2+q^3$, and
$B,C \subset A$ be the
subalgebras generated by 
$x+qa,y+pb$ and by $a,b,F$, respectively. 
Then we
have:
\begin{enumerate}
\item The subalgebra $B$ is a right coideal 
and isomorphic to the coordinate ring of the
nodal cubic, and $C$ is a 
Hopf subalgebra.
\item 
Multiplication in $A$ defines an
isomorphism $C \otimes B \cong A$. 
\item If $ \mathrm{char}(k) \neq 2$ and
$0 \neq I \subset A$ is a Hopf ideal, 
then $B \cap I \neq 0$.
\end{enumerate}
\end{theorem}

Notice that $B \subset A$ is a
coalgebra Galois extension.
The coalgebra in
question is
$A/AB^+$, where $B^+=B \,\mathbin{\cap}\, \mathrm{ker}\,
\varepsilon $. Since $AB^+ \neq B^+A$, this
is not a (Hopf) algebra quotient of $A$. As such, the Hopf subalgebra $C$ in \cref{thm introduction} is  
only isomorphic to $A/AB^+$ as 
coalgebras. 
That $B \subset A$ is a cleft extension
follows immediately from Theorem~1 (2), 
cf.~\cite[Proposition 2.3]{br}. Part (2) of
Theorem~\ref{maintheorem} also implies that $A$ is
free and hence faifhfully flat as a right
$B$-module. Thus $B$ is a \emph{quantum
homogeneous space} in the sense of 
\cite{ms}.

We remark that the Hopf algebra 
$ u_\root (\mathfrak{sl}_2) \otimes 
u_{-1} (\mathfrak{sl}_2) $ is obtained as  
a quotient of $A$ by adding the relation 
$y^2=x^2+x^3=0$. Here, 
$ u_\root (\mathfrak{sl}_2) $ is the
Frobenius-Lusztig
kernel associated to
$U_\root(\mathfrak{sl}_2) $
(see e.g.~\cite[Definition~VI.5.6]{kassel} 
or \cite[Section 3.3.1]{klimyk}) 
and by
a slight abuse of notation,
$u_{-1} (\mathfrak{sl}_2)$ denotes Sweedler's
4-dimensional Hopf algebra.  
The other simple root vector 
in $u_\root(\mathfrak{sl}_2)$ is the class of 
$$
	E := xa-\root ax +\frac{1-\root}{3}
	(a-a^2).
$$

The classification of Hopf algebras in terms of
their Gelfand-Kirillov dimension has been the focus of recent interest,
see~e.g.~\cite{brownzhang,goodearl,liu,wuliu}. 
Prompted by a question from Ken Brown, we show
in \cref{prop properties} that $A$ is
Noetherian of Gelfand-Kirillov dimension
one, but is not semiprime.

In \cite{t}, the construction of quantum homogeneous space structures is extended to more general decomposable plane curves.  Hence an interesting question for future research is whether the results in this note also extend to these curves.

The paper is organised as follows: the next
section contains 
the proof of (1) and (2) in the main theorem, and explains the relation between this paper
and previous results from \cite{kt}.  
Section~\ref{generalcase} contains 
an analogue of Theorem~1 which 
does not require the base field $k$
to contain a primitive third root of unity. 
However, the link to small quantum groups is
lost.
The final section 
discusses the Hopf algebra $A$ from the
perspective of the classification of pointed Hopf
algebras, i.e.~we describe the braided vector
space of skew primitive elements.  This is then
used to prove Theorem~1 (3). We also describe the Nichols algebra of
$A$.   

\section{The cleft extension $B \subset A$}
Throughout this section, let $k$ be a field
containing a primitive third root of unity
$r$, and let $p, q \in k$ be such that $p^2 = q^2 + q^3$.

\subsection{On the origin of relations}
We begin by pointing out that the defining
relations in \cite{kt} can be obtained from
Manin's approach to quantum groups \cite{m}.
To that end, one determines the universal
Hopf algebra coacting by a given formula.

\begin{proposition}
	\label{universal bialgebra}
Assume that $U$ is a Hopf algebra,
$ \rho \colon B \rightarrow B \otimes U$
is a right coaction on 
$B=k[s,t]/\langle s^2+s^3-t^2\rangle$ and that we
have
$$
	\rho (s) = 1 \otimes x + s \otimes a,\quad
	\rho (t) = 1 \otimes y + t \otimes b
$$ 
for elements $a,b,x,y \in U$. 
Then $B$ is a $U$-comodule algebra if and
only if the following relations hold in $U$: 
\begin{gather*}
ba = ab,\qquad 
ya = ay,\qquad 
bx = xb,\qquad
yx = xy,\qquad
by = -yb,\\
	a^2x + axa + xa^2 + a^2-a^3 = 
	x^2a + xax + ax^2 + xa + ax = 0,\\
y^{2} = x^{2} + x^{3},\qquad
b^{2} = a^{3}.
\end{gather*}
\end{proposition}
\begin{proof}
The coaction $ \rho $ turns $B$ into a
$U$-comodule algebra (i.e.~$ \rho $ is
an algebra map) if
and only if the elements 
$ \rho (s)$ and $ \rho (t)$ of $B \otimes U$
satisy the defining relations of $B$,
$$
	\rho (s) \rho (t) = 
	\rho (t) \rho (s), \qquad
	\rho (s)^2 + 
	\rho (s)^3 =  
	\rho (t)^2.
$$
Inserting the explicit formulas of $ \rho (s)$
and $\rho (t)$ shows that the stated commutation
relations in $U$ are sufficient for $\rho $
to be an algebra map. That they are also
necessary follows from the freeness of 
$B \otimes U$ as a $B$-module.
\end{proof}

From now on, we denote by $U$ the universal
Hopf algebra with the above properties. Note
that the coassociativity of $ \rho $ implies
that $a,b$ are necessarily group-like (and
hence invertible) and 
$x,y$ are $(1,a)$- respectively $(1,b)$-primitive.
Thus $U$ is generated as an algebra 
by $a,a^{-1},b,b^{-1},x,y$ satisfying the
above relations. Furthermore, being generated by group-likes and skew primitives, $U$ is a pointed Hopf algebra.
 
The commutation relations in \cref{universal bialgebra} together with the relation $p^2 = q^2 + q^3$ imply that
$$
(y+pb)^2 = (x+qa)^2 +  (x+qa)^3.
$$
Hence, by a slight abuse of notation, we hereby
identify $s$ and $t$ with $x+qa \in U$ and 
$y+pb \in U$, 
respectively, and view $B$ as a subalgebra of $U$. We address the choice of the parameters $(p,q)$ and the connection to \cite{kt} in \cref{subsection parameters}.

\subsection{Reduction of the Hopf algebra}
Our aim is to find 
minimal quotient Hopf algebras of $U$
which contain $B$ as a quantum homogeneous
space. 
The Hopf algebra 
$A$ in \cref{maintheorem} is obtained from $U$
by adding the relations
$b^2=a^3=1$ and $F^3=0$, where
$$F:=xa + (r+1)ax + \frac{r+2}{3} 
\left(a - a^2\right).$$
That the former relation can be added is
straightforward: as $x$ commutes with $b$ and
$y$ commutes with $a$, the element 
$b^2=a^3$ is a central group-like, and the
PBW basis of $U$ presented in \cite{kt}
shows that $B$ embeds into the resulting
quotient.   

The following lemma establishes the
commutation relations between $F$ and the
generators $x+qa,y+pb$ of $B$. This is used
to show that the class of $F^3$ spans a Hopf
ideal in $\tilde A:=U/ \langle a^3-1\rangle$ and that this ideal intersects $B$ trivially. Hence, taking the
quotient by this ideal yields a Hopf
algebra which contains $B$. This is
precisely the Hopf algebra $A$ from
\cref{maintheorem}. 

\begin{lemma}\label{lem z} 
In $\tilde A=U/ \langle a^3-1\rangle$, we have:
\begin{enumerate}
\item The (class of) $F$ satisfies 		 		
		\begin{gather*}
			aF = r^2Fa, \qquad bF = Fb,\qquad
	(y+pb)F=F(y+pb), \\
			(x+qa)F = rF(x+qa)  + 
	\frac{3q+1}{3}(r+2) aF + 
	\frac{r-1}{3} F +  
	\frac{1}{3} (a-1),
	\end{gather*}
	and
	\begin{gather*}
			\Delta(F) = a \otimes F + F \otimes a^2.
		\end{gather*}
\item The element $F^3 \in \tilde A$ is central and
		primitive. 
\item The set 
		$$
		\{a^i b^j F^l 
	\mid 
		i \in \{0,1,2\}, j \in \{0,1\}, 
		l \in \NN\}
		$$
		is a right $B$-module basis of $\tilde A$.
\end{enumerate}
\end{lemma}
\begin{proof}
	The commutation relations and the coproduct 
	of $F$ (and consequently those of $F^3$) are computed in a straightforward
	manner, yielding $(1)$ and $(2)$. 

The basis in $(3)$ of $\tilde A$ as a right
$B$-module is derived from its $k$-linear monomial basis
$$
 \{a^i b^j F^l (x+qa) ^m (y+pb)^n 
 \mid 
 i,n \in \{0,1,2\}, j \in \{0,1\}, 
 l,m \in \NN\}.
 $$
That this is a basis 
follows from Bergman's diamond lemma
\cite{bergman}.  To apply the latter, we first derive a new
presentation of $\tilde A$ in terms of
generators $a,b,F,x+qa$ and
$y+pb$ and consider the lexicographic monomial ordering
with $a <
b < F < x+qa < y+pb$. The complete set of relations satisfied by these generators
consists of: the relations in $(1)$, the relations in 
\cref{universal bialgebra}, and the relations 
$a^3 = 1$ and ${F=xa + (r+1)ax + \frac{r+2}{3}
	\left(a - a^2\right)}$, all 
rewritten in terms of the new generators. Then, 
one has to verify that there are
no ambiguities in the relations. This follows by a straightforward computation, since written in the form of explicit reduction rules, those relations are: 
\begin{gather*}
	ba = ab,\qquad b^{2} = a^{3}=1, \\
	Fa = raF, \qquad Fb = bF,
	\\
	(x+qa)a = - (r+1)a(x+qa) + F + \frac{r+2}{3}((1+3q)a^2 - a), 
	\\
	 (x+qa)b = b(x + qa),\\
	 (x+qa)F = rF(x+qa)  + 
	 \frac{1+3q}{3}(r+2) aF + 
	 \frac{r-1}{3} F +  
	 \frac{1}{3} (a-1),
	 \\
	(y + pb)a = a(y + pb),\qquad
	(y+pb)b = -b(y+pb) +2pb^2,\\
	(y+pb) F = F (y+pb),\qquad
	(y+pb)(x+qa) = (x+qa)(y+pb),
\\
(y+pb)^{2} = (x+qa)^{2} + (x+qa)^{3}. \qedhere
\end{gather*} 

\end{proof}

\begin{remark} An alternative argument is to
observe that the relations presented above
also characterize $\tilde A$ as a quotient of
an iterated Ore extension. Starting from the
group algebra of $\ZZ_3 \times \ZZ_2$
(generated by $a$ and $b$), one adds $F$
followed by $x+qa$ and $y+pb$, subject to
skew-commutation relations given in $\mathcal{R}$. It is well-known that an Ore extension $R[t; \sigma, \delta]$
	is the ring generated over
	$R$ by one additional generator $t$, subject
	to skew-commutation relations $ta =
	\sigma(a)t  +\delta(a)$, for all $a \in R$,
	and that the monomials $t^i$ form a basis as
	an $R$-module (which is a particular case of Bergman's diamond lemma), see e.g.~\cite[\S 1.2]{mr} for more details. Then $\tilde A$ is simply the quotient of such an iterated Ore extension by the relation $(y+pb)^{2} = (x+qa)^{2} + (x+qa)^{3}$.

\end{remark}

We also note, that compared to the
proposition in \cite[p.~657]{kt}, besides the choice of another lexicographic order, the role of the independent variable $xa$ in the monomial basis is played by $F$.

We are now ready to prove the first parts of 
\cref{maintheorem}.
\begin{proposition}
		The subalgebra $B \subset A$ 
is a right coideal, and we have:
\begin{enumerate}
\item The elements $a$, $b$ and $F$ generate a Hopf subalgebra $C \subset A$, such that $C \cong A/AB^+$ as coalgebras, where $B^+:=B \cap
		\mathrm{ker}\,\varepsilon $.
\item $A \cong C \otimes B$ as left $C$-comodules
		and right $B$-modules. In particular, $A$ is free, hence faithfully flat, as a $B$-module.
\end{enumerate}
\end{proposition}
\begin{proof}
	In view of $(1)$ and $(3)$ of \cref{lem z}, the elements $a$, $b$ and $F$ generate a Hopf subalgebra $C$ of $A$ with vector space basis
	$$
	\{a^i b^j F^l \mid 
	i,l \in \{0,1,2\}, j \in \{0,1\} 
	\}.
	$$
The elements
$(x+qa - q)$ and $(y+pb-p)$ generate the left
ideal $AB^+$. This is also a
coideal, so there is a unique coalgebra
structure on $A/AB^+$ for which the quotient
map $A \rightarrow A/AB^+$ is a coalgebra
map. Composing with the embedding 
$C \rightarrow A$
yields a coalgebra map $C \to A/AB^+$ which
is easily seen to be bijective, using the
basis from \cref{lem z}. This proves (1).

	Multiplication in $A$ defines a right
$B$-linear map $C \otimes B \to A$ which is
bijective in view of \cref{lem z}.(3).
The identification $C \cong
A/AB^+$ turns $A$ into a left $C$-comodule 
whose coaction $ \lambda \colon A \rightarrow
C \otimes A$ is the coproduct of $A$ followed 
by the quotient map on the first component.
Since $B$ is a right coideal
subalgebra, it is contained  
in	
$$
	A^{\text{Co }C} = 
	\{b \in A \mid \lambda (b) = 
	1 \otimes b\},
	$$ 
which implies that the 
multiplication  $C \otimes B \rightarrow A$ 
is left $C$-colinear. 
\end{proof}

\subsection{On the parameters $p$ and $q$}
	\label{subsection parameters}
Note that in \cite{kt}, the
choice of $p,q$ is
incorporated in the defining relations of the
Hopf algebra. However, if $A_{(p,q)}$ denotes
the Hopf algebras defined there (with
generators labelled by the subscript $(p,q)$), then 
\begin{gather*}
	x_{(0,0)} \mapsto x_{(p,q)} -qa_{(p,q)},\quad  
	y_{(0,0)} \mapsto y_{(p,q)}-pb_{(p,q)},\quad \\
	a_{(0,0)} \mapsto a_{(p,q)}, \quad 
	b_{(0,0)} \mapsto b_{(p,q)}
\end{gather*}
extends to Hopf algebra isomorphisms
$A_{(0,0)} \rightarrow A_{(p,q)}$. Using these, we
here embed the 
quantum homogeneous spaces studied in \cite{kt} 
all into $A_{(0,0)}$. The latter is simply the
universal Hopf algebra coacting in the given
way on $B$ described in \cref{universal bialgebra}.

In this way, the elements 
$x+qa$ and $y+pb$ generate for any choice of $(p,q)$, with $p^2 = q^2 + q^3$,
a right coideal subalgebra. As $A$-comodule
algebras, these are all isomorphic to each
other (and to the coordinate ring of the
nodal cubic).

\subsection{Relation to small quantum groups}
Note that if we denote
$$
E :=xa  -\root ax+\frac{1-\root}{3}
(a-a^2), \quad K:= a^2,
$$
then the following relations are satisfied
$$
KE=\root^2 EK,\quad
[E,F] = \frac{K-K^{2}}{\root - \root^2}.
$$
These are the defining relations of the quantum universal enveloping algebra $U_\root(\mathfrak{sl}_2)$. Furthermore, observe that $x$ can be written in terms of $E$, $F$ and $K$ as follows
$$
x = \frac{1-r^2}{3}FK + \frac{1-r}{3}EK + \frac{r-r^2}{3}\left(K^2 - K\right).
$$ 
Thus, we obtain a presentation of $A$ as a
quotient of $U_\root(\mathfrak{sl}_2) \otimes 
U_{-1} (\mathfrak{sl}_2)$ (to be able to
specify $U_q( \mathfrak{sl}_2)$ at $q=-1$, we
refer to the non-restricted version using
integral forms, see e.g.~\cite[\S 9.2]{cp}). The image of 
$U_\root(\mathfrak{sl}_2) \otimes 1$ is 
the subalgebra generated by $a$  and $x$ and 
the image of $1 \otimes U_{-1} ( \mathfrak{sl}_2)$ is
the subalgebra generated by $b$ and $y$.  

Adding the relation $E^3=0$ is equivalent to
adding the relation $x^2+x^3=0$ and yields
the small quantum group
$u_{\root}(\mathfrak{sl}_2) \otimes 
u_{-1} (\mathfrak{sl}_2)$. 
Note, finally, that 
expressing $A$ in this way also reveals
the Casimir element
$$
	\Omega := EF + 
	\frac{\root^2 K + \root K^{2}}
	{(\root - \root^2) ^2} = 
	(xa)^2 - a^2x - a^2x^2 + \frac{1}{3},	
$$ 
which is central.

\subsection{Ring-theoretic properties of $A$}
Several recent articles 
classify prime Hopf algebras of
Gelfand-Kirillov dimension one under 
additional assumptions,
see~e.g.~\cite{brownzhang,liu,wuliu}. 
Answering a question of Ken Brown, we remark
that the Hopf algebra $A$ does not fall into
this class:
\begin{proposition}
	\label{prop properties}
	The algebra $A$ is Noetherian of
Gelfand-Kirillov dimension one, but neither 
regular nor semiprime.
\end{proposition}
\begin{proof}
	The algebra $A$ is a finitely generated
module over $B$, which is Noetherian and of Gelfand-Kirillov dimension one. Hence so is $A$ (\cite[Proposition 5.5]{kl}).

As $B$ is the coordinate ring of a singular
curve, evaluation in the singularity defines
a $B$-module of infinite projective dimension.
This is the restriction of the trivial 
$A$-module (the ground field on which $A$
acts via $ \varepsilon $), which therefore
has infinite projective dimension (using that 
$A$ is a free $B$-module). 
	
The Casimir element $\Omega $ has 
minimal polynomial 
	$$
		t^3 - \frac{1}{3} t + \frac{2}{27} = \left(t- \frac{1}{3}\right)^2\left(t + \frac{2}{3}\right).
	$$
	Therefore the ideal $I$ generated by  $\left(\Omega - \frac{1}{3}\right)\left(\Omega + \frac{2}{3}\right)$ is nilpotent and $A$ is not semiprime. \end{proof}
\section{A variant for arbitrary $k$}\label{generalcase}

In the case where the field $k$ does not contain
a primitive third root of unity, the element
$F \in \tilde{A}$ is no longer well-defined and 
the relation to the theory of small quantum groups is
lost. However, in this section we modify the definition of $A$
and obtain a variant of \cref{maintheorem}
that holds for arbitrary fields.  

\subsection{The element $c$}
The following analogue of
\cref{lem z} is straightforward and provides a 
substitute for the element $F^3$:
\begin{lemma}
	\label{lemc}
	In $\tilde A$, we define 
	\begin{align*}
	 c &:= (xa)^3 +2a(xa)^2x + a(xa)^2 -3a^2(xa)x^2 - 2a^2(xa)x -\\
	&\quad  - (xa)x  + ax^2 +ax -2x^2-2x.
	\end{align*}
	Then we have:
	\begin{enumerate}
		\item The element $c$ is central and
		primitive. 
		\item The set 
		$$
		\{a^i b^j(xa)^lc^m
	\mid 
		i,l \in \{0,1,2\}, j \in \{0,1\}, 
		m \in \NN\}
		$$
		is a right $B$-module basis of $\tilde A$.
	\end{enumerate}
\end{lemma}

%

Just as in the previous section, the quotient Hopf
algebra $\tilde A/\tilde A c$ contains $B$ as right
coideal subalgebra and thus provides 
a variant of $A$ defined for general ground
fields. For the remainder of this section, $A$
denotes this quotient.  

We note that the element $x^2 + x^3$ (which
is the same as $y^2$) is another central and
primitive element in $\tilde A$. Note that in the
case where there is a primitive third root of
unity $r$, we have the following relation
$$
c = F^3 + (3r-6) (x^2+x^3).
$$
\subsection{The cleaving map}
Unlike in the previous section, we do not have a
Hopf subalgebra $C \subset A$ isomorphic to 
$A/AB^+$. 
We instead explicitly define a cleaving map,
i.e., a map ${A/AB^+
\rightarrow A}$ which is left $A/AB^+$-colinear and convolution invertible. 
The existence of such a map is equivalent to
$A$ being isomorphic  as left $A/AB^+$-comodules
and right $B$-modules to  $A/AB^+ \otimes B$ (see \cite[Proposition 2.3]{br}). 

\begin{proposition}
	The subalgebra $B \subset A$ generated
	by $x + qa$ and $y + pb$ is a right coideal
	of $A$ and there exists a cleaving map $\gamma \colon A/AB^+ \to A$.
\end{proposition}
\begin{proof}
	The first claim is immediate. 
Consider the canonical projection $\pi \colon A \to
A/AB^+$. 
In view of \cref{lemc}, the set
$$
	\{[a^i b^j(xa)^l]\mid i,l \in \{0,1,2\},j \in \{0,1\}\}
$$ 
is a $k$-linear basis of $A/AB^+$.  
We define a splitting 
$\gamma \colon A/AB^+ \to A$ of $\pi$
by
\begin{gather*}
\gamma([a^i b^j]) = a^ib^j, \qquad
\gamma([a^ib^jxa])=a^i b^j xa,\\
\gamma([a^i b^j(xa)^2]) = a^ib^j\left((xa)^2 - axa(x+qa -q)\right).
\end{gather*} 
It is straightforward to check that $\gamma$
is left $A/AB^+$-colinear on the above mentioned
basis. 

The coradical of $A/AB^+$ is spanned by the
group-like elements. The restriction of
$\gamma$ to the coradical is convolution
invertible, with convolution inverse given by 
$[a^ib^j] \mapsto a^{-i}b^{-j}$.  
Hence, by \cite[Proposition 6.2.2.]{rf}, the map $\gamma$ itself is convolution invertible.
\end{proof}

\begin{remark}
	The cleftness of $A$ also follows from the abstract result \cite[Theorem 1.3 (4)]{masuoka} of Masuoka, which also applies to $\tilde A$. 
\end{remark}

\section{Minimality of $A$}
To prove the minimality of $A$ as stated in
\cref{maintheorem}.(3),
we need to describe its space of skew
	primitive elements. In this section, we resume the notation of Section 2, by denoting $A = \tilde{A}/ \tilde A F^3$.

\subsection{Skew primitives in $\tilde A$ and $A$}
\label{subsection twisted}
We start by proving the following lemma:
\begin{lemma}\label{centprim}
Let $H$ be a Hopf algebra, $h\in H$ be a
central primitive element 
and assume that $H$ is free as a $k[h]$-module.
Then the only skew primitive elements in
the ideal $Hh$ are the scalar multiples of $h$.  
\end{lemma}
\begin{proof}
By the hypothesis on $H$, we can derive a vector space basis of $H$ of the form 
$\{v_jh^n\}_{j\in J,n\in \mathbb{N}}$, 
and we can assume without loss of generality that $v_0=1$.

Let $a \in H$ 
be such that
$$
	\Delta (ah) = 1 \otimes ah + ah \otimes g
$$
for some group-like element $g \in H$.
Using the vector space basis, there are unique 
$w_{jn} \in H$ such that 
$ \Delta (a)=\sum_{j,n} v_jh^n \otimes
w_{jn}$.
Since 
$\varepsilon (h)= 0$, we have 
$$
	a = 
	\sum_j \varepsilon (v_j)w_{j0}=
	\sum_{j,n} \varepsilon
	(w_{jn}) v_jh^n
$$
and we obtain
\begin{align*}
& \quad\sum_{j,n} v_jh^n \otimes w_{jn}h+ 
	v_jh^{n+1} \otimes w_{jn}
	= \Delta (a) \Delta (h) =\\
&= \Delta (ah)
	= 1 \otimes ah +
	\sum_{j,n} 	 
	 \varepsilon (w_{jn}) v_jh^{n+1} 
	\otimes g.
\end{align*}
By considering the first tensor component
and $n=0$,
one observes 
$$
	\sum_{j} v_j \otimes w_{j0}h=
	1 \otimes ah$$
which implies that $w_{00} = a$ and $w_{j0}=0$, for $j>0$. We used the hypothesis that $H$ is a free $k[h]$-module in the fact
that $h$ is not a zero divisor. 
Similarly, for all $j,n \in \mathbb{N} $,
one gets
$$
	w_{jn+1}h
	= 
	\varepsilon (w_{jn})g-w_{jn}.
$$
It follows in particular that  
$$
	w_{jn}=0,\quad \forall j> 0,n \ge 0.
$$
Thus $a$ and hence $ah$ are 
in fact elements of $k[h]$: 
$$
	ah = \sum_n \varepsilon (w_{0n}) 
	h^{n+1}.
$$
As $k[h] \subseteq H$ is  
isomorphic to the universal enveloping
algebra of the 1-dimensional Lie algebra,
the only (skew) primitive elements
are scalar multiples of $h$. 
\end{proof}

We now describe the Yetter-Drinfel'd
module of skew primitives in $\tilde A$. 
Let $P_{(1,g)}(\tilde A)$ denote the vector space of $(1,g)$-skew primitive elements, for a fixed group-like $g\in \tilde A$. Since this always contains $k(g-1)$ as a subspace, we are interested in classifying the quotient spaces 
$$
V_g(\tilde A) := 
P_{(1,g)} / k(g-1),
$$
which is the goal the next proposition.
\begin{proposition}
	\label{twistedprimitives}
The vector spaces 
$V_1(\tilde A),V_a(\tilde A)$ and 
$V_b(\tilde A)$ have linear bases given by 
$$
	V_1(\tilde A) : \{[x^2+x^3],[F^3]\},\quad
	V_a(\tilde A) : \{[x],[axa^2]\},\quad
	V_b(\tilde A) : \{[y]\}.
$$
All other $V_g(\tilde A)$ are trivial. 
\end{proposition}
\begin{proof}
Since $x^2+x^3$ and $F^3$ are central and
primitive, the quotient of $\tilde A$ by the ideal
generated by these two elements is a Hopf
quotient $D$, which is finite-dimensional. The image of any skew
primitive element of $\tilde A$ is skew
primitive in $D$. The group-like and skew
primitive elements in $D$ are 
determined by straightforward computation.
The canonical projection identifies the
group-likes with those of $\tilde A$ itself.
The skew primitives in
the quotient are spanned by the residue
classes of $x,axa^2$ and $y$, together with 
$g-1$, for $g$ group-like. 
 
It remains to determine the skew
primitives contained in the ideal 
$\tilde A(x^2+x^3)+\tilde A F^3$. These are
obtained by a two-fold application of
\cref{centprim}, first with $h=x^2+x^3$ 
and subsequently with $h=F^3$. 
\end{proof}
When $k$ contains a primitive third root of
unity, \cref{twistedprimitives} of course follows as well
from the presentation of $\tilde A$ as a quotient of
$U_r( \mathfrak{sl}_2) \otimes
U_{-1}(\mathfrak{sl}_2)$.

The computation in \cref{twistedprimitives} carries over nicely to $A=\tilde A/\tilde A F^3$. Define $V_g(A)$ analogously as above.

\begin{corollary}
The vector spaces  $V_1(A),V_a(A)$ and 
$V_b(A)$ have linear bases given by 
$$
	V_1(A) : \{[x^2+x^3]\},\quad
	V_a(A) : \{[x],[axa^2]\},\quad
	V_b(A) : \{[y]\}.
$$
All other $V_g(A)$ are trivial. 
\end{corollary}
\begin{proof}
Using the basis from Lemma~\ref{lemc}, we
obtain a splitting of $ \tilde A \rightarrow
A$ as a coalgebra map, which embeds $A$ as
the subcoalgebra spanned over $k$ by 
$\{a^ib^j(xa)^l(x+qa)^r(y+pb)^s \mid 
i,l \in \{0,1,2\},j,s \in \{0,1\},r \in
\mathbb{N} \}$. Hence, $A$ and $\tilde A$ have the same group-like elements and under this embedding, $P_{(1,g)}(A) = P_{(1,g)}(\tilde A) \cap A$.
\end{proof}

\subsection{Proof of Theorem~\ref{maintheorem} (3)}
Now we can complete the proof of our main
theorem. 

\begin{proposition}
Assume $\mathrm{char}(k) \neq 2$, 
$B \subset A$ is as in Theorem~\ref{maintheorem} 
and $ \rho \colon A \rightarrow H$ is a surjective Hopf
algebra map with $ \mathrm{ker}\, \rho \cap B =
0$. Then $ \rho $ is an isomorphism.  
\end{proposition}
\begin{proof}
It is sufficient to verify that $ \rho $ is
injective on each space $P_{(g,h)}(A)$ for fixed group-like
elements $g,h \in A$; according to
\cite[Proposition 4.3.3]{rf} (recall that $A$ is pointed), it follows that
$\rho$ is injective, and thus  
an isomorphism. 

First, the
map $ \rho $ induces a group
homomorphism ${\mathbb{Z} _3 \times \mathbb{Z} _2
\rightarrow G}$,
where $G$ is the group of group-likes in $H$. If
the kernel of this homomorphism is a
non-trivial subgroup, then it contains either $a$ or
$b$. If $ {\rho (b)=1}$, then $by+yb=0$ implies
$ \rho (y) = 0$ as  
$ \mathrm{char}(k) \neq 2$. 
Similarly, if
$ \rho (a)=1$, 
then the commutation relations between $a$
and $x$ yield
$ \rho (x)=0$. In both cases, we conclude
that  
$$
	(y+pb)^2-p^2=
	y^2=
	x^2+x^3 
	\in \mathrm{ker}\, \rho \cap B=0,
$$ 
a contradiction. 
Hence, the induced group homomorphism is injective.
%
Since group-like elements are linearly independent, this means that the restriction of $\rho$ to the coradical of $A$ is injective (the coradical of $A$ is simply the group algebra of $\mathbb{Z} _3 \times \mathbb{Z} _2)$.

Assume now that $ \rho (v)=0$, where
$ v$ is $(g,h)$-primitive. Without loss of
generality, we assume $g=1$ (as we can
replace $v$ by $g^{-1}v$). We now proceed on
a case-by-case basis for possible values of $h$, using the description of the spaces $V_h(A)$ given by \cref{twistedprimitives}.

For $h \neq 1,a,b$,
we have
$V_h(A)=0$ and thus $P_{(1,h)} = k(h-1)$.
Therefore, $\rho(v) = 0$ implies $v = 0$,
since we have already shown that $\rho$ is
injective on the coradical of $A$. For $h = 1$, we have $P_{(1,1)} = k(x^2 + x^3)$, hence $\rho(v) = 0$ would contradict $\ker \rho \cap B = 0$, as seen above. For $h = b$, we have $P_{(1,b)} = k(b-1) \oplus ky$, hence $\rho(v) = 0$ would imply $\rho(y) = \lambda (\rho(b)-1)$ for some $\lambda \in k$ (otherwise $\rho(b-1) = 0$, which we have ruled out). But combining this with $\rho(yb + by) = 0$, yields a contradiction. Similarly, for $h = a$, we have $P_{(1,a)} = k(a-1)\oplus kx \oplus k(axa^2)$. Assume without loss of generality that
$
\rho(axa^2) = \lambda (\rho(a)-1) + \mu \rho(x),
$
for some $\lambda, \mu \in k$ (by conjugating
with $a$ if necessary). Plugging this formula
into $xa^2 + axa + a^2x + a^2 - 1 = 0 = ax^2 +xax + x^2a + ax + xa $ again yields a contradiction.
%
\end{proof}

Completely analogously one proves the
minimality of the variants of $A$ in which
the relation $F^3=0$ is replaced by 
$c + \lambda (x^2+x^3)=0$, for $ \lambda \in k$
(such as the one considered in
Section~\ref{generalcase}), and that these
cover all minimal quotient Hopf algebras of
$U$ containing $B$ as right coideal
subalgebra.

\subsection{The Nichols algebra of $\tilde A$}
From the point of view of the classification of pointed Hopf algebras, it is a natural task to describe the Nichols algebra of a given Hopf algebra, for which one uses the Yetter-Drinfel'd
module of skew primitives \cite{as}. For the sake of completeness, we present the Nichols algebra of $\tilde A$ in this final subsection as a corollary of \cref{subsection twisted}.

Let $B(V)$ denote the Nichols algebra 
of a braided vector space $(V,\chi )$.
If $V=V' \oplus V''$ as braided vector space,
then $B(V) \cong B(V') \otimes B(V'')$. In particular, this
applies to the case of 
$V=V_1(\tilde A) \oplus V_a(\tilde A) \oplus 
V_b(\tilde A)$ with respect to the
Yetter-Drinfel'd braiding $ \chi $. This is the flip $\tau $ in $V_1(\tilde A)$ and $-\tau$ in $V_b(\tilde A)$, yielding a symmetric algebra and an exterior algebra as Nichols algebras for these components, respectively.

The
only non-trivial component to discuss is 
$V_a(\tilde A)$, especially when $k$ does not
contain a primitive third root of unity as 
the braiding on $V_a(\tilde A)$ is then not of
diagonal type: if $u:=[x]$, 
$v:=[axa^2]$, then the Yetter-Drinfel'd
braiding is given by
$$
u \otimes u \mapsto 
v \otimes u,\quad
v \otimes u \mapsto 
v \otimes v, 
$$  
$$
u \otimes v \mapsto -u \otimes u-
v \otimes u,\quad
v \otimes v \mapsto 
-u \otimes v-v \otimes v. 
$$
The Nichols algebra $B(V_a(\tilde A))$ can be computed directly 
by determining the kernels of the quantum
symmetrisers - it  
has the defining relations 
\begin{gather*}
	u^3=v^3=0,\qquad
	u^2v+uvu+vu^2= u^2+uv+v^2=0
\end{gather*}
and is 9-dimensional with a basis given by
$$
	\{1,u,v,u^2,vu,v^2,vuv,v^2u,v^2uv,vuv^2\}.
$$

In total, the Nichols algebra of $\tilde A$ is the tensor product of a polynomial ring in two variables with an exterior algebra with one generator and with the Nichols algebra  $B(V_a(\tilde A))$ described in the previous paragraph.



\end{document}